\documentclass{eptcs}
\providecommand{\event}{ACT2021} 

\usepackage{amsthm}
\usepackage{mathtools}
\usepackage[capitalize]{cleveref}
\usepackage{tikz}
\usepackage{amssymb}
\usepackage{breakurl}
\usepackage{dutchcal}
  \crefformat{enumi}{\##2#1#3}
\usepackage[
backend=biber,
sorting=ynt
]{biblatex}%
  \usetikzlibrary{ 
  	cd,
  	math,
  	decorations.markings,
		decorations.pathreplacing,
  	positioning,
  	shapes,
		shadows,
		shadings,
  	calc,
  	fit,
  	quotes,
  	intersections,
    circuits,
    circuits.ee.IEC
  }
  
  \tikzset{trees/.style={
	inner sep=0, 
	minimum width=0, 
	minimum height=0,
	level distance=.5cm, 
	sibling distance=.5cm,
	edge from parent/.style={shorten <= -2pt, draw, ->},
	grow'=up,
	decoration={markings, mark=at position 0.75 with \arrow{stealth}}
	}
}

    \tikzset{
     oriented WD/.style={
        every to/.style={out=0,in=180,draw},
        label/.style={
           font=\everymath\expandafter{\the\everymath\scriptstyle},
           inner sep=0pt,
           node distance=2pt and -2pt},
        semithick,
        node distance=1 and 1,
        decoration={markings, mark=at position \stringdecpos with \stringdec},
        ar/.style={postaction={decorate}},
        execute at begin picture={\tikzset{
           x=\bbx, y=\bby,
           }}
        },
     string decoration/.store in=\stringdec,
     string decoration={\arrow{stealth};},
     string decoration pos/.store in=\stringdecpos,
     string decoration pos=.7,
	 	 dot size/.store in=\dotsize,
	   dot size=3pt,
	 	 dot/.style={
			 circle, draw, thick, inner sep=0, fill=black, minimum width=\dotsize
	   },
     bbx/.store in=\bbx,
     bbx = 1.5cm,
     bby/.store in=\bby,
     bby = 1.5ex,
     bb port sep/.store in=\bbportsep,
     bb port sep=1.5,
     bb port length/.store in=\bbportlen,
     bb port length=4pt,
     bb penetrate/.store in=\bbpenetrate,
     bb penetrate=0,
     bb min width/.store in=\bbminwidth,
     bb min width=1cm,
     bb rounded corners/.store in=\bbcorners,
     bb rounded corners=2pt,
     bb small/.style={bb port sep=1, bb port length=2.5pt, bbx=.4cm, bb min width=.4cm, 
bby=.7ex},
		 bb medium/.style={bb port sep=1, bb port length=2.5pt, bbx=.4cm, bb min width=.4cm, 
bby=.9ex},
     bb/.code 2 args={
        \pgfmathsetlengthmacro{\bbheight}{\bbportsep * (max(#1,#2)+1) * \bby}
        \pgfkeysalso{draw,minimum height=\bbheight,minimum width=\bbminwidth,outer 
sep=0pt,
           rounded corners=\bbcorners,thick,
           prefix after command={\pgfextra{\let\fixname\tikzlastnode}},
           append after command={\pgfextra{\draw
              \ifnum #1=0{} \else foreach \i in {1,...,#1} {
                 ($(\fixname.north west)!{\i/(#1+1)}!(\fixname.south west)$) +(-
\bbportlen,0) 
  coordinate (\fixname_in\i) -- +(\bbpenetrate,0) coordinate (\fixname_in\i')}\fi 
              \ifnum #2=0{} \else foreach \i in {1,...,#2} {
                 ($(\fixname.north east)!{\i/(#2+1)}!(\fixname.south east)$) +(-
\bbpenetrate,0) 
  coordinate (\fixname_out\i') -- +(\bbportlen,0) coordinate (\fixname_out\i)}\fi;
           }}}
     },
     bb name/.style={append after command={\pgfextra{\node[anchor=north] at 
(\fixname.north) {#1};}}}
  }

 \theoremstyle{plain}
\newtheorem{theorem}{Theorem}[section] 
\newtheorem{proposition}[theorem]{Proposition}

\theoremstyle{definition}
\newtheorem{definition}[theorem]{Definition}

\newtheorem*{axiom*}{Axiom}

\theoremstyle{remark}
\newtheorem{example}[theorem]{Example}

\newtheorem{warning}[theorem]{Warning}

\DeclareSymbolFont{stmry}{U}{stmry}{m}{n}
\DeclareMathSymbol\fatsemi\mathop{stmry}{"23}

\DeclareFontFamily{U}{mathx}{\hyphenchar\font45}
\DeclareFontShape{U}{mathx}{m}{n}{
      <5> <6> <7> <8> <9> <10>
      <10.95> <12> <14.4> <17.28> <20.74> <24.88>
      mathx10
      }{}
\DeclareSymbolFont{mathx}{U}{mathx}{m}{n}
\DeclareFontSubstitution{U}{mathx}{m}{n}
\DeclareMathAccent{\widecheck}{0}{mathx}{"71}

\renewcommand{\ss}{\subseteq}

\DeclareMathOperator{\cod}{cod}

\DeclareMathOperator{\Ob}{Ob}

\newcommand{\const}[1]{\texttt{#1}}
\newcommand{\Set}[1]{\mathsf{#1}}
\newcommand{\ord}[1]{\mathsf{#1}}
\newcommand{\cat}[1]{\mathcal{#1}}
\newcommand{\Cat}[1]{\mathbf{#1}}
\newcommand{\Fun}[1]{\mathrm{#1}}

\newcommand{\id}{\mathrm{id}}
\newcommand{\then}{\mathbin{\fatsemi}}

\newcommand{\To}[1]{\xrightarrow{#1}}

\newcommand{\surj}{\twoheadrightarrow}

\newcommand{\tn}[1]{\textnormal{#1}}

\newcommand{\nn}{\mathbb{N}}

\newcommand{\rr}{\mathbb{R}}

\newcommand{\smset}{\Cat{Set}}
\newcommand{\smcat}{\Cat{Cat}}

\newcommand{\set}{\tn{-}\Cat{Set}}

\definecolor{dgreen}{rgb}{0.0, 0.5, 0.3} 
\definecolor{dyellow}{rgb}{8.0, 0.74, 0}

\newcommand{\LTO}[1]{\overset{\text{#1}}{\bullet}}
\newcommand{\bul}[1][black]{{\color{#1}\ensuremath{\bullet}}}

\newcommand{\yon}{\mathcal{y}}
\newcommand{\poly}[1][]{#1\Cat{Poly}}
\newcommand{\0}{\textsf{0}}
\newcommand{\1}{\tn{\textsf{1}}}
\newcommand{\2}{\tn{\textsf{2}}}

\newcommand{\tri}{\mathbin{\triangleleft}}
\newcommand{\tripow}[1]{^{\lhd #1}}

\newcommand{\qqand}{\qquad\text{and}\qquad}

\newcommand{\cofree}[1]{\cat{C}_{#1}}

\newcommand{\para}{\Cat{Para}}
\newcommand{\oorg}{\mathbb{O}\Cat{rg}}
\newcommand{\learn}{\Cat{Learn}}
\newcommand{\llearn}{\mathbb{L}\Cat{earn}}
\newcommand{\euc}{\Cat{Euc}}

\newcommand{\slens}{\Cat{SLens}}
\newcommand{\coalg}{\tn{-}\Cat{Coalg}}

\newcommand{\tree}{\Set{Tree}}
\newcommand{\Tree}{\Cat{Tree}}

\allowdisplaybreaks
\begin{document}

\title{Learners' Languages}
\def\titlerunning{Learners' Languages}
\author{David I. Spivak \institute{Topos Institute\\Berkeley USA}\email{david@topos.institute}}
\def\authorrunning{David I. Spivak}

\date{\vspace{-.2in}}

\maketitle

\begin{abstract}
In ``Backprop as functor'', the authors show that the fundamental elements of deep learning---gradient descent and backpropagation---can be conceptualized as a strong monoidal functor $\para(\euc)\to\learn$ from the category of parameterized Euclidean spaces to that of learners, a category developed explicitly to capture parameter update and backpropagation. It was soon realized that there is an isomorphism $\learn\cong\para(\slens)$, where $\slens$ is the symmetric monoidal category of simple lenses as used in functional programming.

In this note, we observe that $\slens$ is a full subcategory of $\poly$, the category of polynomial functors in one variable, via the functor $A\mapsto A\yon^A$. Using the fact that $(\poly,\otimes)$ is monoidal closed, we show that a map $A\to B$ in $\para(\slens)$ has a natural interpretation in terms of dynamical systems (more precisely, generalized Moore machines) whose interface is the internal-hom type $[A\yon^A,B\yon^B]$.

Finally, we review the fact that the category $p\coalg$ of dynamical systems on any $p\in\poly$ forms a topos, and consider the logical propositions that can be stated in its internal language. We give gradient descent as an example, and we conclude by discussing some directions for future work.
\end{abstract}

\section{Introduction}

In the paper ``Backprop as functor'' \cite{fong2019backprop}, the authors show that gradient descent and backpropagation---as used in deep learning---can be conceptualized as a strong monoidal functor $L\colon\para(\euc)\to\learn$ from the category of parameterized euclidean spaces to that of learners, a category developed explicitly to capture parameter update and backpropagation. Here, $\para$ is a monad on the category of symmetric monoidal categories. It sends $(\cat{C},I,\otimes)$ to a category with the same objects $\Ob\para(\cat{C})\coloneqq\Ob\cat{C}$, but with hom-sets that include a parameterizing object
\[
\para(\cat{C})(c_1,c_2)\coloneqq\{(p,f)\mid p\in\cat{C}, f\colon c_1\otimes p\to c_2\}/\sim
\]
where the parameterizing object $p$ is considered up to an equivalence relation $\sim$.%
\footnote{The equivalence relation $\sim$ is generated by regarding $(p,f)\sim (p',f')$ if there exists an epimorphism $g\colon p\surj p'$ with $f=g\then f'$. As discussed in Gavranovi\'{c}'s thesis \cite{gavranovic2019compositional}, it is often preferable to dispense with the equivalence relation and instead conceive of $\para(\cat{C})$ as a bicategory. This is what we shall do as well, though we use different 2-morphisms; see \cref{warning.maps}.}
The composite of $c_1\otimes p_1\to c_2$ and $c_2\otimes p_2\to c_3$ in $\para(\cat{C})$ has parameterizing object $p_1\otimes p_2$ and is given by the ordinary composite
\[
c_1\otimes(p_1\otimes p_2) \to c_2\otimes p_2\to c_3.
\]
The domain of the backpropagation functor, $\para(\euc)$, is thus the Para construction applied to the Cartesian monoidal category of Euclidean spaces $\rr^n$ and smooth maps.

But it was soon realized that $\learn$ is in fact also given by a Para construction, namely there is an isomorphism $\learn\cong\para(\slens)$, where $\slens$ is the symmetric monoidal category of simple lenses as used in functional programming. The objects of $\slens$ are sets $\Ob(\slens)=\Ob(\smset)$, but a morphism consists of a pair of functions 
\begin{equation}\label{eqn.slens}
	\slens(A,B)\coloneq\{(f_1,f^\sharp)\mid f_1\colon A\to B,\quad f^\sharp\colon A\times B\to A\}.
\end{equation}
Thus a map $A\to B$ in $\para(\slens)$ consists of a set $P$ and functions $f_1\colon A\times P\to B$ and $f^\sharp\colon A\times B\times P\to A\times P$. The authors of \cite{fong2019backprop} developed this structure in order to conceptualize the compositional nature of deep learning as comprising a parameterizing set $P$ (often called ``the space of weights and biases'') and three functions:
\begin{align}\nonumber
	I&\colon A\times P\to B&\text{implement}\\\nonumber
	U&\colon A\times B\times P\to P&\text{update}\\\label{eqn.learn}
	R&\colon A\times B\times P\to A&\text{request}
\end{align}
The implement function is a $P$-parameterized function $A\to B$, and the update and request functions take a pair $(a,b)$ of ``training data'' and both updates the parameter---e.g.\ by gradient descent---and returns an element of the input space $A$, which is used to train another such function in the network. This last step---the request---is not just found in deep learning as practiced, but is in fact crucial for defining composition.%
\footnote{The reader can check that given only $(P,I,U)\colon A\to B$ and $(Q,J,V)\colon B\to C$, one \emph{can} construct a composed parameter set $P\times Q$, and one \emph{can} construct a composed implement function $A\times P\times Q\to C$, but one \emph{cannot} construct an associative update operation $A\times C\times P\times Q\to P\times Q$. In order to get it, one needs the request function $B\times C\times Q\to B$. By endowing morphisms with the request function, as in $\learn$ \eqref{eqn.learn}, composition and a monoidal structure is easily defined.}

But by this point, the notation $(P,I,U,R)$ of $\learn$ has become heavy and the structure seems to be getting lost. Even knowing $\learn\cong\para(\slens)$ seems ad-hoc since the morphisms \eqref{eqn.slens} of $\slens$ are---to this point---mathematically unmotivated. This is where $\poly$ comes in. 

In this note, we observe that $\slens$ is a full subcategory of $\poly$, the category of polynomial functors in one variable, via the functor $A\mapsto A\yon^A$. Using the fact that $(\poly,\otimes)$ is monoidal closed, we will reconceptualize $\para(\slens)$ in terms of polynomial coalgebras, which can be understood as dynamical systems: machines with states that can be observed as ``output'' and updated based on ``input'' \cite{jacobs2017introduction}. In particular, a morphism $A\to B$ in $\learn$ will be recast as a coalgebra on the internal hom polynomial $[A\yon^A,B\yon^B]$, and we will explain this in terms of dynamics. 

This viewpoint allows us to substantially generalize the construction in $\learn$, a construction which also appears prominently in the theory of open games \cite{ghani2016compositional}. Perhaps more interestingly, it allows us to use the fact that the category $p\coalg$ of dynamical systems on any interface $p\in\poly$ forms a \emph{topos}. A topos is a setting in which one can do dependent type theory and higher-order logic. In fact, the topos of $p$-coalgebras is in some ways as simple as possible: it is not only a copresheaf topos $p\coalg\cong\smset^\cat{C}$ for a certain category $\cat{C}$, but in fact the site $\cat{C}$ is the free category on a directed graph that we'll call $\tree_p$. This makes the logic of $p$-coalgebras---and hence of dynamical systems, learners, and game-players---quite simple. However, the particular graph $\tree_p$ associated to $p$ is highly-structured, and we should find that this structure is inherited by the internal language of $p\coalg$. 

The point is to consider logical propositions that can be stated in the internal language of $p\coalg$ and to use these propositions in order to constrain the behavior of learners and game-players (categorified as discussed above), and of interaction patterns between dynamical systems more generally. For example, ``gradient descent and backpropagation'' is a property we can express in the internal language. Note that the term \emph{language} in the title refers to the internal language of the topos $p\coalg$, which can be thought of as a language for specifying or constraining learning algorithms or dynamic organizational patterns more generally.

\subsection*{Plan for the paper}

In \cref{chap.poly} we will discuss various relevant constructions in the category $\poly$ of polynomial functors in one variable. In particular, we will review its symmetric monoidal closed structure $(\poly,\yon,\otimes, [-,-])$, its composition monoidal structure $(\yon,\tri)$, and the notion of coalgebras. We explain how morphisms in $\learn$ can be phrased in terms of coalgebras on internal hom objects, and we reconstruct $\para(\slens)$ in these terms. 

In \cref{chap.topos_learn}, we first show that the category $p\coalg$ of $p$-coalgebras for the endofunctor $p\colon\smset\to\smset$ is a presheaf topos. We then discuss the internal logic of $p\coalg$, and we conclude by giving several directions for future work.

\subsection*{Notation}

We denote the category of sets by $\smset$; we generally denote sets with upper-case letters $A,B$, etc. Given a natural number $N\in\nn$, we write $\ord{N}\coloneqq\{1,\ldots,N\}$, so $\0=\varnothing$, $\1=\{1\}$, $\2=\{1,2\}$, etc. Given sets $A,B$, we often write $AB\coloneqq A\times B$ to denote their Cartesian product. We will denote polynomials with lower-case letters, $p,q$, etc.

\subsection*{Acknowledgments}
We thank David A.\ Dalrymple, Dai Girardo, Paul Kreiner, David Jaz Myers, and Alex Zhu for useful conversations. We also acknowledge support from AFOSR grant FA9550-20-10348.\goodbreak

\section{Constructions in $\poly$}\label{chap.poly}

In this section we review the category $\poly$, for which \cite{kock2012polynomial} is an excellent reference; we also discuss its symmetric monoidal closed structure. Then we discuss polynomial coalgebras and reconceptualize the category $\learn$ and Gavranovi\'{c}'s bicategorical variant, in that language. 
\subsection{Background on $\poly$ as a monoidal closed category}

For any set $A$, let $\yon^A\colon\smset\to\smset$ be the functor \emph{represented} by $A$; that is, $\yon^A$ applied to a set $S$ is $\smset(A,S)\cong S^A$. In particular, $\yon\coloneqq\yon^\1$ is (isomorphic to) the identity functor $S\mapsto S$ and $\1\coloneqq\yon^\0$ is the constant functor $S\mapsto\1$. Note that $\yon^A(\1)\cong\1^A\cong\1$ for any $A$.

The coproduct of functors $F$ and $G$, denoted $F+G$, is taken pointwise; this means there is a natural isomorphism
\[(F+G)(S)\cong F(S)+G(S)\]
where the coproduct $F(S)+G(S)$ is taken in $\smset$. Similarly, for any set $I$ and functors $F_i$, one for each $i\in I$, their coproduct is computed pointwise
\[
	\left(\sum_{i\in I}F_i\right)(S)\cong\sum_{i\in I}F_i(S).
\]

\begin{definition}
A \emph{polynomial functor} $p$ is any coproduct
\[
	p\coloneqq\sum_{i\in I}\yon^{p[i]}
\]
 of representable functors, where $I\in\smset$ and each $p[i]\in\smset$ are sets. We denote the category of polynomial functors and natural transformations between them by $\poly$.
\end{definition}

We note that if $p=\sum_{i\in I}\yon^{p[i]}$ then $p(\1)\cong I$; hence we can write any $p\in\poly$ in canonical form
\begin{equation}\label{eqn.poly}
	p\cong\sum_{i\in p(\1)}\yon^{p[i]}.
\end{equation}
We refer to each $i\in p(\1)$ as a \emph{position} in $p$ and to each $d\in p[i]$ as a \emph{direction} at $i$. 

\begin{example}
We can consider any set $S$ as a \emph{constant} polynomial $\sum_{s\in S}\yon^\0$.
\end{example}

We can consider a polynomial $p\in\poly$ as a set (or discrete category) $p(\1)$ equipped with a functor $p[-]\colon p(\1)\to\smset$. Then a map of polynomials $\varphi\colon p\to q$ can be identified with a diagram as follows
\begin{equation}\label{eqn.poly_typed}
\begin{tikzcd}
	p(\1)\ar[rr, "\varphi_1"]\ar[rd, bend right, pos=.3, "{p[-]}"']&
	\ar[d, phantom, "\underset{\varphi^\sharp}{\Leftarrow}"]&
	q(\1)\ar[dl, bend left, pos=.3, "{q[-]}"]\\&
	\smset
\end{tikzcd}
\end{equation}
That is, $\varphi$ can be decomposed into a function $\varphi_1\colon p(\1)\to q(\1)$ on positions, and for every $i\in p(\1)$ with $j\coloneqq\varphi_1(i)$, a component function $\varphi^\sharp_i\colon q[j]\to p[i]$ on directions. This follows from the Yoneda lemma and the universal property of coproducts. We will sometimes use this $(\varphi_1,\varphi^\sharp)\colon p\to q$ notation below.

\begin{example}
A morphism $A_1\yon^{A_2}\to B_1\yon^{B_2}$ can be identified with a function $\varphi_1\colon A_1\to B_1$ and a function $\varphi^\sharp\colon A_1\times B_2\to A_2$. That is,
\begin{equation}\label{eqn.lens}
	\poly(A_1\yon^{A_2},\, B_1\yon^{B_2})\cong B_1^{A_1}A_2^{A_1B_2}.
\end{equation}
\end{example}

\begin{proposition}\label{prop.composite}
The composite of polynomial functors $p,q\in\poly$, which we denote $p\tri q$, is again polynomial with formula
\begin{equation}\label{eqn.tri_formula}
p\tri q\cong\sum_{i\in p(\1)}\;\sum_{j\colon p[i]\to q(\1)}\yon^{\sum_{d\in p[i]}q[jd]}
\end{equation}
The composition operation $\tri$ is a (nonsymmetric) monoidal structure on $\poly$, with unit $\yon$.
\end{proposition}
\begin{proof}
See page~\pageref{proof.prop.composite}.
\end{proof}

\begin{example}
If $p=\yon^\2$ and $q=\yon+\1$, then $p\tri q\cong\yon^\2+\2\yon+\1$ whereas $q\tri p\cong \yon^\2+\1$.
\end{example}

\begin{example}
Applying a polynomial $p$ to a set $S$ is given by composition: $p(S)\cong p\tri S$.
\end{example}

\begin{proposition}[The symmetric monoidal category $(\poly,\yon,\otimes)$]\label{prop.otimes}
The category $\poly$ has a symmetric monoidal structure with unit $\yon$ and monoidal product $\otimes$ on objects given by the following formula
\[
	p\otimes q\coloneqq\sum_{i\in p(\1)}\sum_{j\in q(\1)}\yon^{p[i]\times q[j]}
\]
\end{proposition}
\begin{proof}
See page~\pageref{proof.prop.otimes}.
\end{proof}

\begin{proposition}[Internal hom {$[-,-]$}]\label{prop.internal_hom}
The $\otimes$ monoidal structure on $\poly$ is closed; that is, for every $p,q\in\poly$ there is a polynomial
\begin{equation}\label{eqn.internal_hom}
[p,q]\coloneqq\sum_{\varphi\colon p\to q}\yon^{\sum_{i\in p(\1)}q[\varphi_1(i)]}
\end{equation}
for which we have a natural isomorphism
\begin{equation}\label{eqn.curry_adjunction}
\poly(r\otimes p,q)\cong\poly(r, [p,q]).
\end{equation}
\end{proposition}
\begin{proof}
See page~\pageref{proof.prop.internal_hom}.
\end{proof}

\begin{example}\label{ex.slens_ihom}
Given sets $A$ and $B$, we use \cref{eqn.lens,eqn.internal_hom} to compute that the internal hom between $A\yon^A$ and $B\yon^B$ is
\[
	[A\yon^A,B\yon^B]\cong B^A A^{AB}\yon^{AB}.
\]
\end{example}

The counit of the adjunction \eqref{eqn.curry_adjunction} is a natural map $\Fun{eval}\colon p\otimes[p,q]\to q$ called \emph{evaluation}. In very much the same way, it induces two sorts of morphisms we will use later:
\begin{equation}\label{eqn.ihom_two_sorts}
	[p_1,q_1]\otimes [p_2,q_2]\to[p_1\otimes p_2,q_1\otimes q_2]
	\qqand
	[p,q]\otimes [q,r]\to[p,r].
\end{equation}

\subsection{Coalgebras, generalized Moore machines, and learners}

Coalgebras for endofunctors $F\colon\smset\to\smset$ form a major topic of study \cite{arbib1982parametrized,adamek2005introduction,jacobs2017introduction}. In this section we recall the definition and explain the relevance to dynamical systems (generalized Moore machines) and learners.

\begin{definition}[Coalgebra]
Given a polynomial $p$, a \emph{$p$-coalgebra} is a pair $(S,\beta)$ where $S\in\smset$ and $\beta\colon S\to p\tri S$. A \emph{$p$-coalgebra morphism} from $(S,\beta)$ to $(S',\beta')$ consists of a function $f\colon S\to S'$ such that the following diagram commutes:
\begin{equation}\label{eqn.coalg_map}
\begin{tikzcd}
	S\ar[r, "\beta"]\ar[d, "f"']&
	p\tri S\ar[d, "p\tri f"]\\
	S'\ar[r, "\beta'"']&
	p\tri S'
\end{tikzcd}
\end{equation}
We denote the category of $p$-coalgebras and their morphisms by $p\coalg$.
\end{definition}

\begin{proposition}\label{prop.sys}
A $p$-coalgebra $(S,\beta)$ can be identified with a map of polynomials
\begin{equation}\label{eqn.misleading}
S\yon^S\to p.
\end{equation}
\end{proposition}
\begin{proof}
One finds an isomorphism $\poly(S,p\tri S)\cong\poly(S\yon^S,p)$ by direct calculation.
\end{proof}

\begin{warning}\label{warning.maps}
Looking at \cref{prop.sys}, one might be tempted to think that a map of $p$-coalgebras as in \eqref{eqn.coalg_map} can be identified with a commuting triangle
\[
\begin{tikzcd}
	S\yon^S\ar[r, "?"]\ar[rr, bend right]& 
	S'\yon^{S'}\ar[r]&
	 p
\end{tikzcd}
\]
but this is not the case; for one thing, the marked arrow does not arise from a function $f\colon S\to S'$. The point is, \eqref{eqn.misleading} \emph{can be misleading} when it comes to maps, and hence we will depart from the so-called $\para$ construction for 2-cells. For us, the correct sort of map between $p$-coalgebras is the usual one, as shown in \eqref{eqn.coalg_map}.
\end{warning}

\begin{proposition}\label{prop.coalg_lax}
For any $p,q\in\poly$ there is a functor
\[
	p\coalg\times q\coalg\to (p\otimes q)\coalg
\]
making $\bullet\coalg$ a lax monoidal functor $\poly\to\smcat$.
\end{proposition}
\begin{proof}
See page~\pageref{proof.prop.coalg_lax}.
\end{proof}

The relevance of coalgebras to dynamics was of interest in the earliest of references we know of, namely \cite{arbib1982parametrized}, where they are referred to as codynamics. We will proceed with our own terminology.

\begin{definition}[Moore machine]
For sets $A,B$, an \emph{$(A,B)$-Moore machine} consists of
\begin{itemize}
	\item a set $S$, elements of which are called \emph{states},
	\item a function $r\colon S\to B$, called \emph{readout}, and
	\item a function $u\colon S\times A\to S$, called \emph{update}.
\end{itemize}
It is further called \emph{initialized} if it is equipped with an element $s_0\in S$.
\end{definition}

With an initialized $(A,B)$-Moore machine $(S,r,u,s_0)$, we can take any $A$-stream $a\colon\nn\to A$ and produce a $B$-stream $b\colon\nn\to B$ inductively using the formula
\[
	s_{n+1}\coloneqq u(s_n,a_n)
	\qqand
	b_n\coloneqq r(s_n).
\]

\begin{proposition}
An $(A,B)$-Moore machine with states $S$ can be identified with a map of polynomials $S\yon^S\to B\yon^A$, and hence with a $B\yon^A$-coalgebra $S\to B\yon^A\tri S$ by \cref{prop.sys}.
\end{proposition}
\begin{proof}
The identification uses $\varphi_1\coloneqq r$ and $\varphi^\sharp\coloneqq u$.
\end{proof}

Replacing $B\yon^A$ with an arbitrary polynomial $p\in\poly$, we think of $p$-coalgebras as generalized Moore machines. We will refer to them as \emph{$p$-dynamical systems} and call $p$ the \emph{interface}. Mathematically, given $\beta\colon S\to p\tri S$, we also get the two-fold composite
\[S\To{\beta} p\tri S \To{p\tri\beta}p\tri p\tri S\]
and indeed the $n$-fold composite $S\to p\tripow{n}\tri S$ for any $n\in\nn$. The idea is that for every state $s\in S$, we get a position $r(s)\in p(\1)$, and for every direction $d\in p[r(s)]$ there, we get a new state $u(s,d)$. We thus think of $p$ as an interface for the dynamical system: $p(\1)$ says what the world can see about the current state---i.e.\ its outward position $i\coloneqq r(s)$---and $p[i]$ says what sort of forces or inputs the state can be subjected to.

A map of polynomials $\varphi\colon p\to p'$ is a change of interface. We can transform a $p$-dynamical system into a $p'$-dynamical system that has the same set of states. Indeed, simply compose any $S\to p\tri S$ with $\varphi\tri S\colon p\tri S\to p'\tri S$. 

More generally, a map $\varphi\colon p_1\otimes\cdots\otimes p_k\to p'$ allows us to take $k$-many dynamical systems $S_1\to p_1\tri S_1$ through $S_k\to p_k\tri S_k$ and use \cref{prop.coalg_lax} to combine them into a single dynamical system
\[
S\to (p_1\otimes\cdots\otimes p_k)\tri S\To{\varphi} p'\tri S
\]
with interface $p'$ and states $S\coloneqq S_1\times\cdots\times S_k$.

\begin{example}\label{ex.wiring_diagram}
Wiring diagrams are one way of combining dynamical systems as above. \begin{equation}\label{eqn.control_diag}
\varphi=
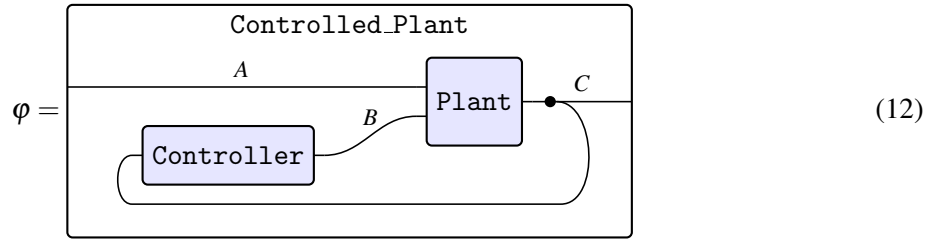
\begin{tikzpicture}[oriented WD, baseline=(B)]
	\node[bb={2}{1}, fill=blue!10] (plant) {\texttt{Plant}};
	\node[bb={1}{1}, below left=-1 and 1 of plant, fill=blue!10]  (cont) {\texttt{Controller}};
	\node[circle, inner sep=1.5pt, fill=black, right=.1] at (plant_out1) (pdot) {};
	\node[bb={0}{0}, inner ysep=20pt, inner xsep=1cm, fit=(plant) (pdot) (cont)] (outer) {};
	\coordinate (outer_out1) at (outer.east|-plant_out1);
	\coordinate (outer_in1) at (outer.west|-plant_in1);
	\begin{scope}[above, font=\footnotesize]
  	\draw (outer_in1) -- node {$A$} (plant_in1);
  	\draw (cont_out1) to node (B) {$B$} (plant_in2);
  	\draw (plant_out1) to node {$C$} (outer_out1);
  	\draw
  		let 
  			\p1 = (cont.south west-| pdot),
  			\p2 = (cont.south west),
  			\n1 = \bby,
  			\n2 = \bbportlen
  		in
  			(pdot) to[out=0, in=0]
  			(\x1+\n2, \y1-\n1) --
  			(\x2-\n2, \y2-\n1) to[out=180, in=180]
  			(cont_in1);
		\end{scope}
	\node[below=0of outer.north] {\texttt{Controlled\_Plant}};
\end{tikzpicture}
\end{equation}
In the wiring diagram \eqref{eqn.control_diag} three boxes are shown: the controller, the plant, and the system; we can consider each as having a monomial interface:
\begin{equation}\label{eqn.basic_diagram}
	\const{Plant}=C\yon^{AB}
	\qquad\quad
	\const{Controller} = B\yon^C
	\qquad\quad
	\const{Controlled\_Plant} = C\yon^A.
\end{equation}
The wiring diagram itself represents a morphism
\[
\varphi\colon C\yon^{AB}\otimes B\yon^C\to C\yon^A
\]
in $\poly$. Defining $\varphi$ requires a function $\varphi_1\colon C\times B\to C$ and a function $\varphi^\sharp\colon C\times B\times A\to A\times B\times C$; the first is projection and the second is an isomorphism. Together these simply say how the wiring diagram shuttles information within the controlled plant. Indeed, the wiring diagram lets us put together dynamics of the controller and the plant to give dynamics for the controlled plant. That is, given Moore machines $S\yon^S\to \const{Plant}$ and $T\yon^T\to \const{Controller}$, we get a Moore machine $ST\yon^{ST}\to \const{Controlled\_Plant}$.

More generally, we can think of transistors in a computer as dynamical systems, and the logic gates, adder circuits, memory circuits, a connected keyboard or monitor, etc.\ each as a wiring diagram comprising these simpler systems.
\end{example}

\begin{example}
In \cref{ex.wiring_diagram} the wiring pattern is fixed, but as we show in  \cite{spivak2020poly}, $\poly$ also supports wiring diagrams for dynamical systems that can change their interaction pattern based on their internal states.
\end{example}

We now come to learners. As mentioned in the introduction, the category $\learn$ from \cite{fong2019backprop} is better understood as a bicategory which we'll denote $\llearn$. Its objects are sets, $\Ob(\llearn)=\Ob(\smset)$, a 1-morphism (learner) from $A$ to $B$ consists of a set $P$ and maps $I\colon A\times P\to B$ and $(R,U)\colon A\times B\times P\to A\times P$, and a 2-morphism---a morphism between learners---is a function $f\colon P\to P'$ making the following squares commute:
\begin{equation}\label{eqn.learn_2morph}
\begin{tikzcd}
	A\times P\ar[r, "I"]\ar[d, "A\times f"']&
	B\ar[d, equal]\\
	A\times P'\ar[r, "I'"']&
	B
\end{tikzcd}
\hspace{.5in}
\begin{tikzcd}[column sep=large]
  A\times B\times P\ar[r, "{(R,U)}"]\ar[d, "A\times B\times f"']&
  A\times P\ar[d, "A\times f"]\\
  A\times B\times P'\ar[r, "{(R',U')}"']&
  A\times P'
\end{tikzcd}
\end{equation}
We denote the category of learners from $A$ to $B$ as $\llearn(A,B)\in\smcat$.

\begin{proposition}\label{prop.learn_coalg}
For sets $A,B$, there is an equivalence of categories
\[\llearn(A,B)\cong[A\yon^A,B\yon^B]\coalg.\]
\end{proposition}
\begin{proof}
See page~\pageref{proof.prop.learn_coalg}.
\end{proof}

We will now give a definition that generalizes the bicategory $\llearn$, give examples, and discuss intuition. In particular, we define a category-enriched operad $\oorg$, which includes $\llearn$ as a full subcategory. 

\begin{definition}[The operad $\oorg$]\label{def.org}
We define $\oorg$ to be the category-enriched operad defined as follows. The objects of $\oorg$ are polynomials: $\Ob(\oorg)\coloneqq\Ob(\poly)$. For objects $p_1,\ldots,p_k, p'$, the category of maps between them is defined by
\[
	\oorg(p_1,\ldots,p_k;p')\coloneqq[p_1\otimes\cdots\otimes p_k, p']\coalg.
\]
For any object $p$, the identity on $p$ is given by the $[p,p]$-coalgebra $\1\to [p,p](\1)\cong\poly(p,p)$ that sends $1\mapsto\id_p$. 

Given objects $p_{1,1},\ldots,p_{1,j_1},\ldots,p_{k,1},\ldots,p_{k,j_k}$, the composition functor
\begin{multline*}
	[p_{1,1}\otimes\cdots\otimes p_{1,j_1},p_1]\coalg
	\times\cdots\times
	[p_{k,1}\otimes\cdots\otimes p_{k,j_k},p_k]\coalg\\
	\times
	[p_1\otimes\cdots\otimes p_k,p']\coalg\to
	[p_{1,1}\otimes\cdots\otimes p_{k,j_k}, p']\coalg
\end{multline*}
is given by repeated application of the maps in \eqref{eqn.ihom_two_sorts} and \cref{prop.coalg_lax}.
\end{definition}

How do we think of a morphism $(S,\beta)\colon (p_1,\ldots,p_k)\to p'$ in $\oorg$? It is a dynamical system which has a set $S$ of states. For every state $s\in S$, we can read out an associated element $\beta_1(s)\colon p_1\otimes\cdots\otimes p_k\to p'$; we can think of this as a wiring diagram as in \cref{ex.wiring_diagram} or a generalization thereof. That is, the current state $s$ dictates an organization pattern $\beta_1(s)$: how outputs of the internal systems are aggregated and output from the outer interface, and how feedback from outside is distributed internally. 

But so far, this is only the readout of $\beta$. What's an input? An input to this system consists of a tuple of outputs $i\coloneqq(i_1,\ldots,i_k)\in p_1(\1)\times\cdots\times p_k(\1)$, one output for each of the internal systems, together with an input $d\in p'[\beta_1(i)]$ to the outer system. 

Imagine you're the officer in charge of an organization: you're in charge of the system by which your employees and other resources are arranged, how they send information to each other and the outside world, and how the feedback from the outside world is disbursed to the employees and resources. You see what they do, you see how the world responds, and you update your internal state and hence the system itself, however you see fit. In this image, you as the officer are playing the role of $(S,\beta)$, i.e.\ a morphism in $\oorg$. But even a simple logic gate or adder circuit in a computer---something that doesn't have a changing internal state or update how resources are connected---counts as a morphism in $\oorg$. Again, the only difference in that case is that the state set $S\cong\1$, the way the internal resources are connected---is unchanged by inputs.

\begin{example}\label{ex.0ary}
For any operad, there is an algebra of $0$-ary morphisms. In the case of $\oorg$, this algebra sends $p\mapsto p\coalg$, the category of dynamical systems on $p$, since the unit of $\otimes$ is $\yon$ and $[\yon,p]\cong p$.
\end{example}

Next we'll give a mathematical language for describing dynamical systems as in \cref{ex.0ary} as well as the generalized learners (or officers) described above.

\section{Toposes of learners}\label{chap.topos_learn}

We ended the previous section by defining the (category-enriched) operad $\oorg$ and explaining how it generalizes the bicategory $\llearn$. In this section we mainly discuss the internal language for each learner. That is, given $p,p'\in\Ob(\poly)=\Ob(\oorg)$, where perhaps $p=p_1\otimes\cdots\otimes p_k$, we discuss the category $\oorg(p;p')$ of such learners. 

Our first job is to show that every such category is a topos; this will give us access to the Mitchell-Benabou language and Kripke-Joyal semantics---the so-called \emph{internal language} of the topos and its interpretation \cite{macLane1992sheaves}. We then explain the sorts of things---propositions---that one can express in this language, e.g.\ the proposition ``I will follow the gradient descent algorithm'' is a particular case.

\subsection{The topos of $p$-coalgebras}

In this section, we show that for any polynomial $p$, there is a category $\cofree{p}$, called the \emph{cofree category on $p$}, for which we can find an equivalence
\[
p\coalg\cong\cofree{p}\set
\]
between $p$-coalgebras and functors $\cofree{p}\to\smset$. In fact, the category $\cofree{p}$ is free on a graph, making it quite easy to understand in certain respects.%
\footnote{The name ``cofree category'' comes from the fact that---up to isomorphism---comonoids in $\poly$ are categories; see \cite{ahman2016directed}. So we're really taking the cofree comonoid on $p$.}

Following \cite{nlab:tree}, we define a \emph{rooted tree} to be a graph $T$ whose free category has an initial object, called the \emph{root}; the idea is that for any node $n$, there is exactly one path from the root to $n$. We denote the nodes of $T$ by $T_0$, the root by $\text{root}_T\in T_0$, and for any node $n\in T_0$ we denote the set of arrows emanating from $n$ by $T[n]$. Note that at the target $n'$ of any arrow $a\in T[n]$, there sits another rooted tree with root $n'$; we denote this tree by $\cod_T(a)$.

\begin{definition}[The graph $\tree_p$ of $p$-trees]\label{def.ptrees}
For a polynomial $p\in\poly$, define a \emph{$p$-tree} to be a tuple $(T,\phi_1,\phi^\sharp)$, where $T$ is a rooted tree, $\phi_1\colon T_0\to p(\1)$ is a function called the \emph{position} function, and $\phi^\sharp_n$ is a bijection
\[
\phi^\sharp_n\colon p[\phi_1(n)]\To{\cong}T[n]
\]
for each node $n\in T_0$, identifying the set of branches in the tree $T$ at node $n$ with the set of directions in the polynomial $p$ at the position $\phi_1(n)$. 

We denote by $\tree_p$ the graph whose vertex set is the set of $p$-trees, and for which an arrow $a\colon T\to T'$ is a branch $a\in T[\text{root}_T]$ with $T'=\cod_T(a)$.
\end{definition}

\begin{example}
If $p=\yon^A$ for a nonempty set $A$ then there is only one $p$-tree: each node has the unique label $p(\1)\cong\1$ and $A$-many branches.

If $p=\{\text{go}\}\yon^\1+\{\text{stop}\}\yon^\0\cong\yon+\1$ then counting the number of nodes gives a bijection between set of $p$-trees and the set $\nn\cup\{\infty\}$
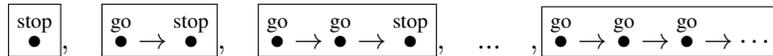
\[
  \fbox{$\LTO{stop}$},\quad
  \fbox{$\LTO{go}\to\LTO{stop}$},\quad
  \fbox{$\LTO{go}\to\LTO{go}\to\LTO{stop}$},\quad...\quad,
  \fbox{$\LTO{go}\to\LTO{go}\to\LTO{go}\to\cdots$}
\]
\end{example}


\begin{theorem}\label{thm.coalg_copresheaf}
For any polynomial $p$ there is an equivalence of categories
\[
p\coalg\cong\Tree_p\set
\]
where $\Tree_p$ is the free category on the graph $\tree_p$ of $p$-trees.
\end{theorem}
\begin{proof}
See page~\pageref{proof.thm.coalg_copresheaf}.
\end{proof}

\begin{figure}
\[
\begin{tikzpicture}
\node[draw] (p1) {
\begin{tikzcd}[column sep=small, ampersand replacement=\&]
	\bul[dgreen]\ar[rr, bend left]\ar[loop left]\&\&
	\bul[dyellow]\ar[dl, bend left]\ar[ll, bend left]\\\&
	\bul[red]
\end{tikzcd}
};
\node[draw, right=of p1] {
\begin{tikzpicture}[trees, scale=1.2,
  level 1/.style={sibling distance=20mm},
  level 2/.style={sibling distance=10mm},
  level 3/.style={sibling distance=5mm},
  level 4/.style={sibling distance=2.5mm},
  level 5/.style={sibling distance=1.25mm}]
  \node[dgreen] (a) {$\bullet$}
    child {node[dgreen] {$\bullet$}
    	child {node[dgreen] {$\bullet$}
    		child {node[dgreen] {$\bullet$}
  				child {node[dgreen] {$\bullet$}
    				child {}
    				child {}
    			}
  				child {node[dyellow] {$\bullet$}
    				child {}
    				child {}
    			}
  			}
    		child {node[dyellow] {$\bullet$}
					child {node[dgreen] {$\bullet$}
      			child {}
      			child {}
     			}
    			child  {node[red] {$\bullet$}}
  			}
    	}
    	child {node[dyellow] {$\bullet$}
    		child {node[dgreen] {$\bullet$}
  				child {node[dgreen] {$\bullet$}
    				child {}
    				child {}
    			}
  				child {node[dyellow] {$\bullet$}
    				child {}
    				child {}
    			}
  			}
    		child  {node[red] {$\bullet$}}
    	}
    }
    child {node[dyellow] {$\bullet$}
    	child {node[dgreen] {$\bullet$}
    		child {node[dgreen] {$\bullet$}
  				child {node[dgreen] {$\bullet$}
    				child {}
    				child {}
    			}
  				child {node[dyellow] {$\bullet$}
    				child {}
    				child {}
    			}
  			}
    		child {node[dyellow] {$\bullet$}
					child {node[dgreen] {$\bullet$}
      			child {}
      			child {}
     			}
    			child  {node[red] {$\bullet$}}
  			}
  		}
  		child {node[red] {$\bullet$}
  		}
  	}
  ;
\end{tikzpicture}
};
\end{tikzpicture}
\]
\caption{Left: a dynamical system, i.e.\ coalgebra, for the polynomial $p\coloneqq\{\bul[dgreen],\bul[dyellow]\}\yon^\2+\bul[red]\cong\2\yon^\2+\1$. Right: the $p$-tree corresponding to the node $\bul[dgreen]$.
}
\end{figure}

\subsection{The internal language of $p\coalg$}

For any category $\cat{C}$, the category $\cat{C}\set$ of functors $\cat{C}\to\smset$ forms a topos. In particular, this means that mathematicians have already developed a language and logic that faithfully represents the structures of $\cat{C}\set$, and we can import it wholesale; see \cite[Chapter 7]{fong2019seven} or \cite[Chapter VI]{macLane1992sheaves}. Now that we know from \cref{thm.coalg_copresheaf} that $p\coalg$ is a topos for any $p\in\poly$, we are interested in corresponding language for the topos $\llearn(A,B)=[A\yon^A,B\yon^B]\coalg$ of learners, for any sets $A,B$; hence the title ``learners' languages.'' However since most of the relevant abstractions work more generally for $p\coalg$,  we'll mainly work there.

Not assuming the reader knows topos theory, we will proceed as though we are defining the few relevant concepts from scratch, when in actuality we are merely ``reading them off'' from the established literature. For example \cref{def.logical_prop} simply unpacks the topos-theoretic definition of a logical proposition as a subobject of the terminal object in the topos $p\coalg$.

\begin{definition}\label{def.logical_prop}
A \emph{logical proposition (about $p$-coalgebras)} is defined to be a set $P\ss\tree_p$ of $p$-trees satisfying the condition that if $T\in P$ is a tree in $P$, then for any direction $d\in T[\text{root}_T]$, the tree $\cod_T(d)\in P$ is also a tree in $P$.
\end{definition}

\cref{prop.easy_subobjects} gives us an easy way to construct logical propositions about $p$-coalgebras, and hence learners. Namely, it says if we put a condition on the $p$-positions that can show up as labels, and if we put a condition on the codomain map (how directions in the tree lead to new positions), we get a logical proposition. Of course, these aren't the only ones, but they form a nice special case.

Recall from \cref{def.ptrees} that a $p$-tree is a rooted tree $T$ equipped with a position function $\phi_1\colon T_0\to p(\1)$; we elide the bijections (earlier denoted $\phi_n^\sharp\colon T[n]\cong p[\phi_1(n)]$) in what follows.

\begin{proposition}\label{prop.easy_subobjects}
Given $p\in\poly$, suppose given subsets
\[
	Q\ss p(\1)
	\qqand
	R\ss\prod_{i\in Q}\prod_{d\in p[i]}Q.
\]
Then the following set of trees is a logical proposition:
\[
	P_Q^R\coloneqq
	\{T\in\tree_p\mid
	\forall(i:T_0). \phi_1(i)\in Q\wedge
	\forall(d:T[i]).\cod_T(d)\in (R i d)
	\}.
\]
\end{proposition}
\begin{proof}
The result is immediate from \cref{def.logical_prop}.
\end{proof}

\begin{example}[Gradient descent]
The gradient descent, backpropagation algorithm used by each ``neuron'' in a deep learning architecture can be phrased as a logical proposition about learners. The whole learning architecture is then put together as in \cite{fong2019backprop}, or as we've explained things above, using the operad $\oorg$ from \cref{def.org}.\goodbreak

So suppose a neuron is tasked with learning a function $\rr^m\to\rr^n$, and it has a parameter space $\rr^k$, i.e.\ we are given a smooth function $f\colon\rr^k\times\rr^m\to\rr^n$. We will define a corresponding logical proposition using \cref{prop.easy_subobjects}. Define $p\in\poly$ by
\[
	p\coloneqq
	[\rr^m\yon^{\rr^m},\rr^n\yon^{\rr^n}]\cong
	\sum_{g\colon\rr^m\yon^{\rr^m}\to\rr^n\yon^{\rr^n}}\yon^{\rr^m\times\rr^n}.
\]
Define $Q\ss\{(g_1,g^\sharp)\mid g_1\colon\rr^m\to\rr^n, g^\sharp\colon\rr^m\times\rr^n\to\rr^m\}$ by saying that $g_1(x)$ must be of the form $f(a,x)$ for some $a\in\rr^k$ in the parameter set and that $g^\sharp_x$ is given by ``pulling back gradient vectors'' using the map on cotangent spaces defined by composing with the derivative of $g_1$ at $x$, in the usual way.

Now given $(g_1,g^\sharp)\in Q$, we continue with the setup of \cref{prop.easy_subobjects} by defining $R(g_1,g^\sharp)\colon \rr^m\times\rr^n\to Q$ to say how the learner updates its current parameter value $a\in\rr^k$ given an input-output pair; this again is specified by the deep learning algorithm. Typically, it uses a loss function to calculate a cotangent vector at $f(a,x)$ which is passed back to a cotangent vector at $a$, and a dual vector of some ``learning rate'' $\epsilon$ is traversed.

The details are important for implementation, but not for understanding the idea. The idea is that as long as we say what sorts of maps are allowed (smooth maps with reverse derivatives) and how they update, we have defined a logical proposition.
\end{example}

The logical propositions that come from \cref{prop.easy_subobjects} are very special. More generally, one could have a logical proposition like ``whenever I receive two red tokens within three seconds, I will wait five seconds and then send either three blue tokens or two blues and six reds.'' As long as this behavior has the ``whenever'' flavor---more precisely as long as it satisfies the condition in \cref{def.logical_prop}---it will be a logical proposition in the topos.

\subsection{Future work}\label{sec.future_work}

There are many avenues for future work. One is to give more syntactic language---beyond the logical symbols $\texttt{true},\texttt{false},\wedge,\vee,\Rightarrow,\neg,\forall,\exists$ that exist in any topos---for building logical propositions in the $p\coalg$ toposes specifically. Another is to understand various modalities in these toposes.

The sort of morphisms between toposes that seem to arise most naturally in this context are not the usual kind---adjoint functors $\cat{E}\leftrightarrows\cat{E}'$ for which the left adjoint preserves all finite limits, called \emph{geometric morphisms}---but instead adjoint functors $\cat{E}\leftrightarrows\cat{E}'$ for which the left adjoint preserves all \emph{connected limits}. Thus another avenue for future research is to consider how logical and type-theoretic statements move between toposes that are connected in this way.

\appendix
\section{Proofs}

\begin{proof}[Proof of \cref{prop.composite}]\label{proof.prop.composite}
It is well-known that composition of functors is a monoidal operation, so it suffices to see that the polynomial \eqref{eqn.tri_formula} is the composite of functors $p,q$. To show this, we use the fact that for any set $A$ we have a bijection $\yon^A\cong\prod_{a\in A}\yon$ to calculate the composite
\begin{align*}
	p\tri q&\cong
	\sum_{i\in p(\1)}\prod_{d\in p[i]}\yon\;\tri\;\sum_{j\in q(\1)}\prod_{e\in q[j]}\yon\\&\cong
	\sum_{i\in p(\1)}\prod_{d\in p[i]}\sum_{j\in q(\1)}\prod_{e\in q[j]}\yon\\&\cong
	\sum_{i\in p(\1)}\sum_{j\colon p[i]\to q(\1)}\prod_{d\in p[i]}\prod_{e\in q[j(d)]}\yon
	\cong\sum_{i\in p(\1)}\;\sum_{j\colon p[i]\to q(\1)}\yon^{\sum_{d\in p[i]}q[jd]}
\end{align*}
where the first isomorphism is \eqref{eqn.poly}, the second is substitution, the third is the distributive law, and the fourth is properties of exponents.
\end{proof}

\begin{proof}[Proof of \cref{prop.otimes}]\label{proof.prop.otimes}
With the formula given, it is clear that the $\otimes$-operation is associative (up to isomorphism), and that $\yon$, which has $\yon(\1)\cong\1$ and $\yon[1]\cong\yon$, is a unit. One can also check that the formula is functorial in $p$ and $q$, completing the proof.
\end{proof}

\begin{proof}[Proof of \cref{prop.internal_hom}]\label{proof.prop.internal_hom}
The natural isomorphism is given by rearranging terms:
\begin{align*}
	\poly(r\otimes p,q)&\cong
	\prod_{k\in r(\1)}\prod_{i\in p(\1)}\sum_{j\in q(\1)}\prod_{e\in q[j]}r[k]\times p[i]\\&\cong
	\prod_{k\in r(\1)}\sum_{\varphi_1\colon p(\1)\to q(\1)}\prod_{i\in p(\1)}\prod_{e\in q[\varphi_1(i)]}r[k]\times p[i]\\&\cong
	\prod_{k\in r(\1)}\sum_{\varphi_1\colon p(\1)\to q(\1)}\left(\prod_{i\in p(\1)}\prod_{e\in q[\varphi_1(i)]}p[i]\right)\times\left(\prod_{i\in p(\1)}\prod_{e\in q[\varphi_1(i)]}r[k]\right)\\&\cong
	\prod_{k\in r(\1)}\sum_{\varphi\colon p\to q}\prod_{i\in p(\1)}\prod_{e\in q[\varphi_1(i)]}r[k]\\&\cong
	\poly\left(r,\sum_{\varphi\colon p\to q}\yon^{\sum_{i\in p(\1)}q[\varphi_1(i)]}\right)\cong
	\poly(r,[p,q]).
\end{align*}
In order, these isomorphisms are given by: unfolding the definition of morphisms in $\poly$, distributivity, products commuting with products, definition of morphisms in $\poly$, rules of exponents, and \cref{eqn.internal_hom}'s definition of $[p,q]$, respectively.
\end{proof}

\begin{proof}[Proof of \cref{prop.coalg_lax}]\label{proof.prop.coalg_lax}
We need to give not only the functor $\lambda\colon p\coalg\times q\coalg\to (p\otimes q)\coalg$, for any $p,q\in\poly$ but also a functor $\{1\}\to\yon\coalg$, which we can identify with a $\yon$-coalgebra; we take the latter to be the unique function $\1\to\yon\tri\1$. For the former, one could proceed abstractly using the fact that there is a duoidal structure on $\poly$
\[
	(p\tri s)\otimes(q\tri t)\to (p\otimes q)\tri(s\otimes t).
\]
Indeed, since for sets $S,T$ we have $S\otimes T\cong S\times T$, the result will follow from the properties of duoidal structures (applied in the case where $s\coloneqq S$ and $t\coloneqq T$ are constant polynomials). However, for the reader's convenience, we will give the map $\lambda\colon p\coalg\times q\coalg\to (p\otimes q)\coalg$ more explicitly. 

Given $\beta\colon S\to p\tri S$ and $\gamma\colon T\to q\tri T$, we define a function
\begin{align*}
  ST&\to(p\otimes q)\tri (ST)\cong\sum_{i\in p(\1)}\sum_{j\in q(\1)}(ST)^{p[i]\times q[j]}\\
  (s,t)&\mapsto
\big(i\coloneqq\beta_1(s),j\coloneqq\gamma_1(t), (d,e)\mapsto (\beta^\sharp_i(d),\gamma^\sharp_j(e)\big).
\end{align*}
This is natural in $S,T$, which makes $\lambda$ a functor for any $p,q$. One can check that all the axioms of a lax monoidal functor are verified, in the sense that the required diagrams commute up to natural isomorphism.
\end{proof}

\begin{proof}[Proof of \cref{prop.learn_coalg}]\label{proof.prop.learn_coalg}
On one hand, an object in $\llearn(A,B)$ as described in \eqref{eqn.learn} consists of a set $P$ and functions $A\times P\to B$ and $A\times B\times P\to A$ and $A\times B\times P\to P$. On the other hand, we have $[A\yon^A,B\yon^B]\cong B^A A^{AB}\yon^{AB}$ by \cref{ex.slens_ihom}, so a coalgebra $P\to[A\yon^A,B\yon^B]\tri P$ consists of a function $P\to B^A$, a function $P\to A^{AB}$, and a function $P\to P^{AB}$. The two descriptions can be identified by currying. The $[A\yon^A,B\yon^B]\coalg$ morphisms
\[
\begin{tikzcd}
	P\ar[r]\ar[d]&
	B^AA^{AB}P^{AB}\ar[d]\\
	P'\ar[r]&
	B^AA^{AB}(P')^{AB}
\end{tikzcd}
\]
are easily seen to coincide with those shown in \eqref{eqn.learn_2morph}.
\end{proof}

\begin{proof}[Proof of \cref{thm.coalg_copresheaf}]\label{proof.thm.coalg_copresheaf}
It is well-known that $\cat{C}\set$ is equivalent to the category of discrete opfibrations over $\cat{C}$ via the category-of-elements construction. When $\cat{C}$ is free on a graph $G$, the category of elements for any functor $h\colon\cat{C}\to\smset$ is also free on a graph, say $H$. In this case the opfibration can be identified with a graph homomorphism $\pi\colon H\to G$ with the property (``opfib'') that for any vertex $h\in H$, the function on arrows $H[h]\To{\cong} G[\pi(h)]$ induced by $\pi$ is a bijection. Under this correspondence, a morphism $h\to h'$ of copresheaves is identified with a graph homomorphism $f\colon H\to H'$ for which $\pi=\pi'\circ f$.

Thus we have reduced to showing that there is an equivalence between $p\coalg$ and the category of those graph homomorphisms $\pi\colon H\to\tree_p$ that have the opfib property. Suppose given a $p$-coalgebra $\beta\colon S\to p\tri S$; it includes a function $\beta_1\colon S\to p(\1)$ and for each $s\in S$ a function $\beta_s^\sharp\colon p[\beta_1(s)]\to S$. We define the corresponding graph $G_{S,\beta}$ to have vertex set $S$, and each $s\in S$ to have $p[\beta_1(s)]$-many outgoing arrows; the target of each outgoing arrow $d\in p[\beta_1(s)]$ is defined to be $\beta_s^\sharp(d)$. The graph homomorphism $\pi\colon G_{S,\beta}\to\tree_p$ is defined inductively: for any $s\in S$, the $p$-tree $\pi(s)$ has root labeled $\beta_1(s)$, and for each outgoing branch $d\in p[\beta_1(s)]$ the target vertex is assigned the label $\beta_1(s')$, where $s'\coloneqq\beta_s^\sharp(d)$, and for each outgoing branch $d'\in p[\beta_1(s')]$ the target vertex is assigned the label $\beta_1(s'')$ where $s''\coloneqq \beta_{s'}^\sharp(d')$, and so on. It is clear that $\pi$ satisfies the opfib property, since it assigns to each vertex $s$ in the graph $G_{S,\beta}$ a vertex in $\tree_p$ (the $p$-tree $\pi(s)$) with the same set $p[\beta_1(s)]$ of outgoing arrows.

Conversely, given a graph homomorphism $\pi\colon G\to\tree_p$ with the opfib property, let $S_G$ be the set of vertices in $G$. The required coalgebra map $\beta\colon S_G\to p\tri S_G$ consists of a function $\beta_1\colon S_G\to p(\1)$ and a function $\beta^\sharp_s\colon p[\beta_1(s)]\to S_G$ for every $s\in S_G$. We take the function $\beta_1$ to send vertex $s$ to the root label $\phi(\text{root}_{\pi(s)})$ for tree $\pi(s)$. Since we have a bijection $p[\beta_1(s)]\cong G[s]$, we can take $\beta^\sharp_s$ to simply be the target function $G[s]\to S_G$ for the graph $G$. 

It is a straightforward calculation to check that these two constructions are mutually inverse, and to check that graph homomorphisms over $\tree_p$ correspond bijectively to morphisms of $p$-coalgebras.
\end{proof}

\printbibliography
\end{document}